\documentclass[12pt, twoside]{article}
\usepackage{amsmath,amsthm,amssymb}
\usepackage{times}
\usepackage{enumerate}
\usepackage{bm}
\usepackage[all]{xy}
\usepackage{mathrsfs}
\usepackage{amscd}
\usepackage{color}
\def\titlerunning#1{\gdef\titrun{#1}}
\makeatletter
\def\author#1{\gdef\autrun{\def\and{\unskip, }#1}\gdef\@author{#1}}
\def\address#1{{\def\and{\\\hspace*{18pt}}\renewcommand{\thefootnote}{}%
		\footnote {#1}}%
	\markboth{\autrun}{\titrun}}
\makeatother
\def\email#1{e-mail: #1}

\def\keywords#1{\par\medskip
	\noindent\textbf{Keywords.} #1}

\newtheorem{theorem}{Theorem}[section]
\newtheorem{corollary}[theorem]{Corollary}
\newtheorem{lemma}[theorem]{Lemma}
\newtheorem{proposition}[theorem]{Proposition}
\theoremstyle{definition}
\newtheorem{definition}[theorem]{Definition}
\newtheorem{remark}[theorem]{Remark}
\newtheorem{example}[theorem]{Example}
\numberwithin{equation}{section}

\xyoption{all}

\def \C {\mathbb{C}}

\def \D {\mathcal{D}}

\def \a {\alpha }
\def \b {\beta}

\def \De {\Delta}
\def \la {\lambda}
\def \La {\Lambda}
\def\w {\omega}
\def\Om{\Omega}

\def\pa{\partial}

\begin{document}
	\baselineskip=17pt
	
	\titlerunning{On $L^{2}$-harmonic forms of complete almost K\"{a}hler manifold}
	\title{On $L^{2}$-harmonic forms of complete almost K\"{a}hler manifold}
	
	\author{Teng Huang}
	
	\date{}
	
	\maketitle
	
	\address{T. Huang: School of Mathematical Sciences, University of Science and Technology of China; CAS Key Laboratory of Wu Wen-Tsun Mathematics, University of Science and Technology of China, Hefei, Anhui, 230026, P. R. China; \email{htmath@ustc.edu.cn;htustc@gmail.com}}
	\begin{abstract}
	In this article, we study the $L^{2}$-harmonic forms on the complete $2n$-dimensional almost K\"{a}her manifold $X$.  We observe that the $L^{2}$-harmonic forms can decomposition into Lefschetz powers of primitive forms. Therefore we can extend vanishing theorems of $d$(bounded) (resp. $d$(sublinear)) K\"{a}hler  manifold proved by Gromov (resp. Cao-Xavier, Jost-Zuo) to almost K\"{a}hlerian case, that is, the spaces of all harmonic $(p,q)$-forms on $X$ vanishing unless $p+q=n$. We also give a lower bound on the spectra of the Laplace operator to sharpen the Lefschetz vanishing theorem on $d$(bounded) case.
	\end{abstract}
	\section{Introduction}
	A differential form $\a$ in a Riemannian manifold $(X,g)$ is called bounded with respect to the metric $g$ if the $L_{\infty}$-norm of $\a$ is finite, namely,
	$$\|\a\|_{L_{\infty}}=\sup_{x\in X}|\a(x)|<\infty.$$
	By definition, a $k$-form $\a$ is said to be $d$(bounded) if $\a=d\b$, where $\b$ is a bounded $(k-1)$-form. It is obvious that if $X$ is compact, then every exact form is $d$(bounded). However, when $X$ is not compact, there exist smooth differential forms which are exact but not $d$(bounded). For instance, on $\mathbb{R}^{n}$, $\a=dx^{1}\wedge\cdots\wedge dx^{n}$ is exact, but it is not $d$(bounded) \cite{CY,Gro}.  Let’s recall some concepts introduced in \cite{CX,Hitchin, JZ}.
	\begin{definition}
		A differential form $\a$ on a complete non-compact Riemannian manifold $(X,g)$ is called $d$(sublinear) if there exist a differential form $\b$ and a number $c>0$ such that $\a=d\b$ and 
		$$\ \|\b(x)\|_{L_{\infty}}\leq c(1+\rho_{g}(x,x_{0})),$$  
		where $\rho_{g}(x,x_{0})$ stands for the Riemannian distance between $x$ and a base point $x_{0}$ with respect to $g$.
	\end{definition}
	Let $(X, g)$ be a Riemannian manifold and $\pi:(\tilde{X},\tilde{g})\rightarrow(X,g)$ be the universal covering with $\tilde{g}=\pi^{\ast}g$. A form $\a$ on $X$ is called $\tilde{d}$(bounded) (resp. $\tilde{d}$(sublinear)) if $\pi^{\ast}\a$ is a $d$(bounded) (resp. $d$(sublinear)) form on $(\tilde{X},\tilde{g})$.
	In geometry, various notions of hyperbolicity have been introduced, and the typical examples are manifolds with negative curvature in suitable sense \cite{CY}. The starting point for the present investigation is Gromov's notion of K\"{a}hler hyperbolicity \cite{Gro}. Extending Gromov’s terminology, Cao-Xavier \cite{CX} and Jost-Zuo \cite{JZ} proposed the K\"{a}hler parabolicity. 

	Let $(X^{2n},\w)$ be a closed symplectic manifold. Let $J$ be an $\w$-compatible almost complex structure, i.e., $J^{2}=-id$, $\w(J\cdot, J\cdot)=\w(\cdot,\cdot)$,  and $g(\cdot,\cdot)=\w(\cdot, J\cdot)$ is a Riemannian metric on $X$. The triple $(\w, J, g)$ is called an almost K\"{a}hler structure on $X$. Notice that any one of the pairs $(\w, J)$, $(J, g)$ or $(g, \w)$ determines the other two. An almost-K\"{a}hler metric $(\w,J,g)$ is K\"{a}hler if and only if $J$ is integrable. For the symplectic case, inspired by K\"{a}hler geometry, Tan-Wang-Zhou \cite{Ked,TWZ} gave the definition of a symplectic hyperbolic (resp. parabolic) manifold. 
	\begin{definition}
		A closed almost K\"{a}hler manifold $(X,\w)$ is called symplectic hyperbolic (resp. parabolic) if the lift $\tilde{\w}$ of $\w$ to the universal covering $(\tilde{X},\tilde{\w})\rightarrow(X,\w)$ is $d$(bounded) (resp. $d$(sublinear)) on $(\tilde{X},\tilde{\w})$.
	\end{definition}
	\begin{example}
(1) Let $(X^{2n},\w)$ be a closed symplectic manifold. If $[\w]$ is aspherical and $\pi_{1}(X)$ is hyperbolic then, $\w$ is hyperbolic. In particular, if $X^{2n}$ admits a Riemannian metric of negative sectional curvature then, $\w$ is hyperbolic, see \cite[Corollary 1.13]{Ked}.\\
(2) Let $X^{2n}$ be a closed manifold of non-positive sectional curvature. If $X^{2n}$ is homeomorphic to a symplectic manifold, then $X^{2n}$ is symplectically parabolic \cite{CX,JZ}.
	\end{example}
	Tan-Wang-Zhou proved that if $(X^{2n},\w)$ is a closed symplectic parabolic manifold which satisfies the Hard Lefschetz Condition (HLC), then the spaces of $L^{2}$-harmonic forms $\mathcal{H}^{k}_{(2)}(\tilde{X})$ on the universal space $\tilde{X}$ are zero unless $k=n$ \cite{HTan,TWZ}. The Hard Lefschetz Condition is necessary in Tan-Wang-Zhou's theorem. Hind-Tomassini \cite{HT} constructed a $d$(bounded) complete almost K\"{a}hler manifold $X$ satisfying $\mathcal{H}^{1}_{(2)}(X)\neq\{0\}$ by using methods of contact geometry.
	
	In \cite{Gro}, Gromov developed $L^{2}$-Hodge theory for  K\"{a}hler manifolds, proving an $L^{2}$-Hodge decomposition Theorem for $L^{2}$-forms. As a consequence, for a complete and $d$(bounded) K\"{a}hler manifold $X$, denoting by $\mathcal{H}^{k}_{(2)}$, respectively $\mathcal{H}^{p,q}_{(2)}$, the space of $\De_{d}$-harmonic $L^{2}$-forms of degree $k$, respectively $\De_{d}=2\De_{\bar{\pa}}$-harmonic $L^{2}$-forms of bi-degree $(p,q)$, he showed that $\mathcal{H}^{k}_{(2)}\cong\bigoplus_{p+q=k}\mathcal{H}^{p,q}_{(2)}$; furthermore, denoting by $n=\dim_{\C}X$, that $\mathcal{H}^{k}_{(2)}=\{0\}$, for all $k\neq n$ and hence $\mathcal{H}^{p,q}_{(2)}=\{0\}$, for all $(p,q)$ such that $p+q\neq n$. Gromov \cite{Gro} also gave a lower bound on the spectra of the Laplace operator $\De_{d}:=dd^{\ast}+d^{\ast}d$ on $L^{2}$-forms $\Om^{p,q}(X)$ for $p+q\neq n$ to sharpen the Lefschetz vanishing theorem. The main purpose in this article is to extend the Gromov's results to almost K\"{a}hlerian case.
	\begin{theorem}\label{T1}[=Theorem \ref{T4} and \ref{T3}]
		Let $(X,J,\w)$ be a complete $2n$-dimensional almost K\"{a}hler  manifold with a $d$(sublinear) symplectic form $\w$.  Then 
		$$\mathcal{H}_{(2);J}^{p,q}(X)=\{0\}$$
		unless $k:=p+q=n$, where 
		$$\mathcal{H}_{(2);J}^{p,q}(X):=\{\a\in\Om^{p,q}_{(2);J}(X):\De_{d}\a=0\}.$$ 
		In particular, if $\w$ is $d$(bounded), i.e., there exists a bounded $1$-form $\theta$ such that $\w=d\theta$, then any $\a\in \Om_{J}^{p,q}\cap\Om^{k}_{0}$ on $X$ of degree $k:=p+q\neq n$ satisfies the inequality
		\begin{equation*}
		c_{n,k}\|\theta\|^{-2}_{L_{\infty}}\|\a\|^{2}_{L^{2}(X)}\leq\|d\a\|^{2}_{L^{2}(X)}+\|d^{\ast}\a\|^{2}_{L^{2}(X)},
		\end{equation*}
		where $c_{n,k}>0$ is a constant which depends only on $n,k$.
	\end{theorem}
Suppose that $(X^{2n},\w)$ is a  complete K\"{a}hler manifold. Let $\a_{k}$ be a $k$-form in $X^{2n}$. We denote $\a_{k}:=\sum_{p+q=k}\a_{p,q}$, where $\a_{p,q}\in\Om^{p,q}(X)$. We have
$$\langle\De_{d}\a_{k},\a_{k}\rangle_{L^{2}(X)}=\sum_{p+q=k}\langle\De_{d}\a_{p,q},\a_{p,q}\rangle_{L^{2}(X)}\Rightarrow \mathcal{H}^{k}_{(2)}(X)=\bigoplus_{p+q=k}\mathcal{H}^{p,q}_{(2);J}(X).$$	
Following Theorem \ref{T1}, we get some well-known results proved by Gromov \cite{Gro}, Cao-Xavier \cite{CX} and Jost-Zuo\cite{JZ}.
	\begin{corollary}[=Corollary \ref{C2} and \ref{C1}]
		Let $(X^{2n},\w)$ be a complete K\"{a}hler  manifold with a $d$(sublinear) K\"{a}hler form $\w$.  Then 
		$$\mathcal{H}_{(2)}^{k}(X)=\{0\}$$
		unless $k=n$. In particular, if $\w$ is $d$(bounded), i.e., there exists a bounded $1$-form $\theta$ such that $\w=d\theta$, then any $\a\in \Om^{k}_{0}$ on $X$ of degree $k\neq n$ satisfies the inequality
		\begin{equation*}
		c_{n,k}\|\theta\|^{-2}_{L_{\infty}}\|\a\|^{2}_{L^{2}(X)}\leq\|d\a\|^{2}_{L^{2}(X)}+\|d^{\ast}\a\|^{2}_{L^{2}(X)},
		\end{equation*}
		where $c_{n,k}>0$ is a constant which depends only on $n,k$.
	\end{corollary}

	\section{$L^{2}$-Hodge theory}
	We recall some basic on $L^{2}$ harmonic forms \cite{Car,Dodziuk}.  Let $X$ be a smooth manifold of dimension $n$, let $\Om^{k}(X)$ and $\Om^{k}_{0}(X)$ denote the smooth $k$-forms on $X$ and the smooth $k$-forms with compact support on $X$, respectively. We assume now that $X$ is endowed with a Riemannian metric $g$. Let $(,)$ denote the pointwise inner product on $\Om^{k}(X)$ given by $g$.  The global inner product is  defined
	$$\langle\a,\b\rangle=\int_{X}(\a,\b)dvol,$$
	where $dvol$ is the Riemannian volume form of metric $g$.
	
	 We also write $|\a|^{2}=(\a,\a)$, $\|\a\|^{2}=\int_{X}|\a|^{2}dvol$, and let $$\Om^{k}_{(2)}(X)=\{\a\in\Om^{k}(X):\|\a\|^{2}<\infty\}.$$
	Denote by $(A_{(2)}^{k}(X),d)$ the sub-complex of $(\Om^{k}(X),d)$ formed by differential forms $\a$ such that both $\a$ and $d\a$ are in $L^{2}$. Then the reduced $L^{2}$-cohomology group of degree $k$ of $X$ is defined as
	$$H^{k}_{(2)}(X)=A^{k}_{(2)}(X)\cap\ker d/\overline{(d\Om^{k-1}_{(2)}(X))}.$$
	We recall the following 
	\begin{lemma}(\cite[Lemma 1.1 A]{Gro})\label{L9}
		Let $(X,g)$ be a complete Riemannian manifold of dimension $n$ and let $\eta$ be an $L^{1}$-form on $X$ of degree $n-1$, that is
		$$\int_{X}|\eta|<\infty.$$
		Assume that also the differential $d\eta$ is also $L^{1}$. Then
		$$\int_{X}d\eta=0.$$ 
	\end{lemma}
	Let $d^{\ast}$ denote the adjoint operator of the differential operator $d$ with respect to $g$. The Laplacian operator is given by $\De_{d}=dd^{\ast}+d^{\ast}d:\Om^{k}(X)\rightarrow\Om^{k}(X)$. A $k$-form $\a\in\Om^{k}_{(2)}(X)$ is called $L^{2}$-harmonic form if $\De_{d}\a=0$. It is well known that $\a$ is $L^{2}$-harmonic if only if $d\a=0$ and $d^{\ast}\a=0$. We denote by 
	$$\mathcal{H}^{k}_{(2)}(X)=\{\a\in\Om^{k}_{(2)}(X):\De_{d}\a=0\}$$ 
	the space of $L^{2}$-harmonic $k$-forms on $X$. We have the Hodge-de Rham-Kodaira orthogonal decomposition of $\Om^{k}_{(2)}(X)$
	$$\Om^{k}_{(2)}(X)=\mathcal{H}^{k}_{(2)}(X)\oplus\overline{d(\Om^{k-1}(X))}\oplus\overline{d^{\ast}(\Om^{k+1}(X))},$$
	where $\overline{d(\Om^{k-1}(X))}$ and $\overline{d^{\ast}(\Om^{k+1}(X))}$ are closure of $d(\Om^{k-1}(X))$ and $d^{\ast}(\Om^{k+1}(X))$ with respect to $L^{2}$-norm respectively (see \cite{Car} and \cite[1.1.C.]{Gro}).
	We have the following 
	\begin{lemma}(\cite[Lemma 2.2]{HT})\label{L6}
		Let $(X,g)$ be a complete Riemannian manifold of dimension $n$ and let $\a\in\Om^{2}_{(2)}(X)$. Denote by
		$$\a=\a_{H}+\la+\mu$$
		the Hodge decomposition of $\a$, where $\a_{H}\in\mathcal{H}^{k}_{(2)}$, $\la\in\overline{d\Om^{k-1}(X)}$, $\mu\in\overline{d^{\ast}\Om^{k+1}(X)}$. Then\\
		(1) $d\la=0$,\\
		(2) If $d\a=0$, then $\mu=0$.
	\end{lemma}
	
	\section{Vanishing theorems}
	As we derive estimates in this section, there will be many constants which appear. Sometimes we will take care to bound the size of these constants, but we will also use the following notation whenever the value of the constants are unimportant. We write $\a\lesssim\b$ to mean that $\a\leq C\b$ for some positive constant $C$ independent of certain parameters on which $\a$ and $\b$ depend. The parameters on which $C$ is independent will be clear or specified at each occurrence. We also use $\b\lesssim\a$ and $\a\approx\b$ analogously.
	\subsection{$L^{2}$-harmonic forms of bi-degree}
	We review and point out certain special structures of differential forms on symplectic manifolds. Let $(X,\w)$ be a closed symplectic manifold of dimension $2n$. Using the symplectic form $\w=\sum\frac{1}{2}\w_{ij}dx^{i}\wedge dx^{j}$, the
	Lefschetz operator $L:\Om^{k}\rightarrow\Om^{k+1}$ and the dual Lefschetz operator $\La:\Om^{k}\rightarrow\Om^{k-2}$ are defined acting on a $k$-form $\a_{k}$ by
	\begin{equation*}
	L(\a_{k})=\w\wedge\a_{k},\ \ \La(\a_{k})=\frac{1}{2}(\w^{-1})^{ij}i_{\pa_{x_{i}}}i_{\pa_{x_{j}}}\a_{k}.
	\end{equation*}
	\begin{definition}\label{L1}(\cite{Tseng-Yau} Definition 2.1)
		A differential $k$-from $B_{k}$ with $k\leq n$ is called $primitive$, i.e., $B_{k}\in P^{k}(X)$, if it satisfies the two equivalent conditions: (i) $\La B_{k}=0$; (ii) $L^{n-k+1}B_{k}=0$.
	\end{definition}
	We will make use of the Weil relation for primitive $k$-forms $B_{k}$, see \cite{Tseng-Yau} Equation (2.19):
	\begin{equation}\label{E5}
	\ast\frac{1}{r!}L^{r}B_{k}=(-1)^{\frac{k(k+1)}{2}}\frac{1}{(n-k-r)!}L^{n-k-r}\mathcal{J}(B_{k}),
	\end{equation}
	where $$\mathcal{J}=\sum_{p,q}(\sqrt{-1})^{p-q}\Pi^{p,q}$$
	projects a $k$-form onto its $(p,q)$ parts time the multiplicative factor $(\sqrt{-1})^{p-q}$. 
	
	Recall the symplectic star operator, $\ast_{s}:\Om^{k}(X)\rightarrow\Om^{2n-k}(X)$ defined by
	\begin{equation}\label{E11}
	\begin{split}
	\a\wedge\ast_{s}\b&=(\w)^{-1}(\a,\b)dvol\\
	&=\frac{1}{k!}(\w^{-1})^{i_{1}j_{1}}(\w^{-1})^{i_{2}j_{2}}\cdots(\w^{-1})^{i_{k}j_{k}}\a_{i_{1}i_{2}\cdots i_{k}}\b_{j_{1}j_{2}\cdots j_{k}}\frac{\w^{n}}{n!},\\
	\end{split}
	\end{equation}
	for any two $k$-forms $\a,\b\in\Om^{k}$. This definition is in direct analogy with the Riemannian Hodge star operator where here $\w^{-1}$ has replaced $g^{-1}$. Notice, however, that $\ast_{s}$ as defined in (\ref{E11}) does not give a positive-definite local inner product, as $\a\wedge\ast_{s}\a$ is $k$-symmetric. Thus, for instance, $\a_{k}\wedge\ast_{s}\a_{k}$ for $k$ odd \cite{Tseng-Yau}. The symplectic star operator permits us to consider $\La$ and $d^{\La}$ as the symplectic adjoints of $L$ and $d$, respectively. Specifically, we have the relations \cite{Mat,Yan}  $\La=\ast_{s}L\ast_{s}$,
	and \cite{Bry}
	\begin{equation}\label{E12}
	d^{\La}:=[d,\La]=(-1)^{k+1}\ast_{s}d\ast_{s},
	\end{equation}
	acting on $\a_{k}\in\Om^{k}$. Thus we easily find that $d^{\La}$ squares to zero, that is, $d^{\La}d^{\La}=-\ast_{s}d^{2}\ast_{s}=0$.
	
	Let $(X,\w,g)$ be a closed almost K\"{a}hler manifold. We can used the metric $g$ to define the Hodge star operator. The dual Lefschetz operator $\La$ is then just the adjoint of $L$, $\La=(-1)^{k}\ast L\ast$. The $d^{\La}$ operator is related via the Hodge star operator defined with respect to the compatible metric $g$ by the relation, see \cite[Lemma 2.9]{Tseng-Yau},
	$$d^{\La}=(-1)^{k+1}\ast\mathcal{J}^{-1}d\ast\mathcal{J}^{-1}=-{\ast}\mathcal{J}^{-1}d\mathcal{J}\ast.$$
	Let $X$ be a $2n$-dimensional manifold (without boundary) and $J$ be a smooth almost-complex structure on $X$. There is a natural action of $J$ on the space $\Om^{k}(X,\C):=\Om^{k}(X)\otimes\C$, which induces a topological type decomoposition
	$$\Om^{k}(X,\C)=\bigoplus_{p+q=k}\Om^{p,q}_{J}(X,\C),$$
	where  $\Om^{p,q}_{J}(X,\C)$ denotes the space of complex forms of type $(p,q)$ with respect to $J$ \cite{HMT}. If $k$ is even, $J$ also acts on $\Om^{k}(X)$ as an involution. The space  $\Om^{k}(X)$ of real smooth differential $k$-forms has a type decomposition:
	$$\Om^{k}(X)=\bigoplus_{p+q=k}\Om^{p,q}_{J}(X),$$
	where 
	$$\Om^{p,q}_{J}(X)=\{\a\in\Om^{p,q}_{J}(X,\C)\oplus\Om^{q,p}_{J}(X,\C):\a=\bar{\a}\}.$$
	
	We denote by 
	$$\mathcal{H}^{p,q}_{(2);J}(X):=\{\a\in\Om^{p,q}_{(2);J}(X):\De_{d}\a=0 \} $$
	the space of $L^{2}$-harmonic forms of bi-degree $(p,q)$. Here $$\Om^{p,q}_{(2);J}(X):=\{\a\in\Om^{p,q}_{J}(X):\|\a\|_{L^{2}(X)}<\infty\}.$$

	In a closed symplectic manifold $(X,\w)$, Tseng-Yau considered the symplectic cohomology group  $H^{k}_{d+d^{\La}}$ which are just the symplectic version of well-known cohomologies in complex geometry already studied by Kodaira-Spencer \cite{KS}. They also defined a four-order differential operator as follows
	$$\De_{d+d^{\La}}=dd^{\La}(dd^{\La})^{\ast}+\la(d^{\ast}d+d^{\La\ast}d^{\La}),$$
	where $d^{\La_{\ast}}=([d,\La])^{\ast}=[L,d^{\ast}]=\ast d^{\La}\ast$ (see \cite[Euqation (2.25)]{Tseng-Yau}).  In \cite{Tseng-Yau}, the authors proved that there exists a Lefschetz decomposition for symplectic cohomology group  $H^{k}_{d+d^{\La}}$.
	
	In a complete symplectic manifold $(X,\w)$, we only consider a two-order differential operator as follows
	$$\D_{d+d^{\La}}=d^{\ast}d+d^{\La\ast}d^{\La}.$$
	We denote by 
	$$\mathcal{H}^{p,q}_{(2);d+d^{\La}}:=\{\a\in\Om^{p,q}_{(2);J}:\D_{d+d^{\La}}\a=0 \}$$
	the space of $L^{2}$ $d+d^{\La}$-harmonic $(p,q)$-forms. We also denote by $$P\mathcal{H}^{p,q}_{(2);d+d^{\La}}:=\{\a\in P^{p,q}_{(2);J}:\D_{d+d^{\La}}\a=0  \}$$
	the space of $L^{2}$ $d+d^{\La}$-harmonic $(p,q)$-forms,
where $P^{p,q}_{(2);J}:=\ker\La\cap\Om^{p,q}_{(2);J}$ is the space of the primitive $L^{2}$ $(p,q)$-forms.
	\begin{lemma}\label{L8}
	For any $\a_{p,q}\in\Om^{p,q}_{J}$, we have the identity
	$$\|d^{\La}\a_{p,q}\|^{2}=\|d^{\ast}\a_{p,q}\|^{2}.$$
	In particular,
	$$\mathcal{H}^{p,q}_{(2);J}(X)=\mathcal{H}^{p,q}_{(2);d+d^{\La}}(X).$$
\end{lemma}
\begin{proof}
	Noting that $\mathcal{J}^{2}=(-1)^{k}$ acting on a $k$-form. We then have $$d\ast\mathcal{J}^{-1}\a_{p,q}=d\ast(-1)^{k}\mathcal{J}\a_{p,q}=(-1)^{k}(\sqrt{-1})^{p-q}d\ast\a_{p,q}.$$
	Therefore,
	\begin{equation*}
	\|d^{\La}\a_{p,q}\|^{2}=\|\mathcal{J}^{-1}d\ast\mathcal{J}^{-1}\a_{p,q}\|^{2}=\|d\ast\mathcal{J}^{-1}\a_{p,q}\|^{2}=\|d\ast\a_{p,q}\|^{2}=\|d^{\ast}\a_{p,q}\|^{2}.
	\end{equation*}  
	Therefore, $d^{\La}\a_{p,q}=0$ if only if $d^{\ast}\a_{p,q}=0$. Suppose that $\a_{p,q}\in\mathcal{H}^{p,q}_{(2);d+d^{\La}}(X)$. Then following Lemma \ref{L3}, we get $d^{\La}\a_{p,q}=0$ and $d\a_{p,q}=0$. Hence $d^{\ast}\a_{p,q}=0$, i.e., $\a_{p,q}\in\mathcal{H}^{p,q}_{(2);J}(X)$.
\end{proof}	
	We follow the method of Gromov's \cite{Gro} to choose a sequence of cutoff functions $\{f_{\varepsilon}\}$ satisfying the following conditions:\\
	(i) $f_{\varepsilon}$ is smooth and takes values in the interval $[0,1]$, furthermore, $f_{\varepsilon}$ has compact support.\\
	(ii) The subsets $f^{-1}_{\varepsilon}\subset X$, i.e., of the points $x\in X$ where $f_{\varepsilon}(x)=1$ exhaust $X$ as $\varepsilon\rightarrow0$.\\
	(iii) The differential of $f_{\varepsilon}$ everywhere bounded by $\varepsilon$,
	$$\|df_{\varepsilon}\|_{L_{\infty}}=\sup_{x\in X}|df_{\varepsilon}|\leq\varepsilon.$$
	Thus one obtains another useful
	\begin{lemma}\label{L3}
		If an $L^{2}$ $(p,q)$-form $\a$ is $\mathcal{D}_{d+d^{\La}}$-harmonic form, then $d\a=0$, $d^{\La}\a=0$.
	\end{lemma}
	\begin{proof}
		We want to justify the integral identity
		$$\langle\D_{d+d^{\La}}\a,\a\rangle=\langle d\a,d\a\rangle+\langle d^{\La}\a,d^{\La}\a\rangle$$
		If $d\a$ and $d^{\La}\a$  are $L^{2}$ (i.e., square integrable on $X$), then this follows by Lemma \ref{L9}. To handle the general case we cutoff $\a$ and obtain by a simple computation
		\begin{equation*}
		\begin{split}
		0&=\langle\D_{d+d^{\La}}\a,f^{2}_{\varepsilon}\a\rangle\\
		&=\langle d\a,d(f^{2}_{\varepsilon}\a)\rangle+\langle d^{\La}\a,d^{\La}(f^{2}_{\varepsilon}\a)\rangle\\
		&=\langle d\a,f^{2}_{\varepsilon}d\a\rangle+\langle d\a,2f_{\varepsilon}df_{\varepsilon}\wedge\a\rangle\\
		&+\langle d^{\La}\a,f^{2}_{\varepsilon}d^{\La}\a\rangle+\langle d^{\La}\a,2f_{\varepsilon}df_{\varepsilon}\wedge(\La\a)\rangle-\langle d^{\La}\a,\La(2f_{\varepsilon}df_{\varepsilon}\wedge\a)\rangle\\
		&=I_{1}(\varepsilon)+I_{2}(\varepsilon),\\
		\end{split}
		\end{equation*}
		where
		\begin{equation*}
		\begin{split}
		|I_{1}(\varepsilon)|&=\langle d\a,f^{2}_{\varepsilon}d\a\rangle+\langle d^{\La}\a,f^{2}_{\varepsilon}d^{\La}\a\rangle\\
		&=\int_{X}f^{2}_{\varepsilon}(|d\a|^{2}+|d^{\La}\a|^{2})\\
		\end{split}
		\end{equation*}
		and
		\begin{equation*}
		\begin{split}
		|I_{2}(\varepsilon)|&\leq|\langle d\a,2f_{\varepsilon}df_{\varepsilon}\wedge\a\rangle|+|\langle d^{\La}\a,2f_{\varepsilon}df_{\varepsilon}\wedge(\La\a)\rangle|+|\langle d^{\La}\a,\La(2f_{\varepsilon}df_{\varepsilon}\wedge\a)\rangle|\\
		&\lesssim\int_{X}|df_{\varepsilon}|\cdot|f_{\varepsilon}|\cdot|\a|(|d\a|+|d^{\La}\a|).\\
		\end{split}
		\end{equation*}
		Then we choose $f_{\varepsilon}$ such that $|df_{\varepsilon}|^{2}<\varepsilon f_{\varepsilon}$ on $X$ and estimate $I_{2}$ by Schwartz inequality. Then
		$$|I_{2}(\varepsilon)|\lesssim \varepsilon\|f_{\varepsilon}\a\|_{L^{2}(X)}\big{(}\int_{X}f^{2}_{\varepsilon}(|d\a|^{2}+|d^{\La}\a|^{2})\big{)}^{\frac{1}{2}},$$
		and hence $|I_{1}|\rightarrow0$ for $\varepsilon\rightarrow0$.
	\end{proof}
	\begin{proposition}\label{P7}
		Let $(X^{2n},J,\w)$ be a complete almost K\"{a}hler  manifold. We then have a decomposition for the space of the $L^{2}$-harmonic forms of bi-degree $(p,q)$:
		$$\mathcal{H}^{p,q}_{(2);J}(X)=\bigoplus_{r\geq\max\{0,k-n\}} L^{r}P\mathcal{H}^{p-r,q-r}_{(2);J}(X),$$
		where $k:=p+q$.
	\end{proposition}
	\begin{proof}
		Let $\a_{k}$ be a $L^{2}$ $\mathcal{D}_{d+d^{\La}}$-harmonic form of bi-degree $(p,q)$ on $X$, $k:=p+q$. Following primitive decomposition formula, see \cite[Proposition 1.2.30]{Huy} or \cite[Charp VI. (5.15)]{Dem}, we can denote $$\a_{k}=\sum_{r\geq\max\{0,k-n\}}L^{r}\b_{k-2r},$$ 
		where $\b_{k-2r}\in P^{k-2r}$. Following \cite[Lemma 2.3, 2.10]{Tseng-Yau}, we have
		$$[\D_{d+d^{\La}},L]=[d^{\ast}d+d^{\La_{\ast}}d^{\La},L]=[d^{\ast},L]d+d^{\La_{\ast}}[d^{\La},L]=0.$$
		Therefore,
		\begin{equation}\label{E1}
		0=\D_{d+d^{\La}}\a_{k}=\sum_{r\geq\max\{0,k-n\}}L^{r}(\D_{d+d^{\La}}\b_{k-2r}).
		\end{equation}
		Noting that the operator $\mathcal{D}_{d+d^{\La}}$ communicates with $\La$, see \cite[Lemma 3.7]{Tseng-Yau}, 
		$$[\mathcal{D}_{d+d^{\La}},\La]=[d^{\ast}d+d^{\La_{\ast}}d^{\La},\La]=d^{\ast}[d,\La]+[d^{\La_{\ast}},\La]d^{\La}=0.$$
		Then $\La\mathcal{D}_{d+d^{\La}}\b_{k-2r}=0$, i.e., $\mathcal{D}_{d+d^{\La}}\b_{k-2r}\in P^{k-2r}$. Using the fact (\cite[Proposition 1.2.30]{Huy})
			$$ \Om^{k}(X)=\bigoplus_{r\geq\max\{0,k-n\}}L^{r}P^{k-2r}(X).$$
	Hence following (\ref{E1}), we then have
		$$\D_{d+d^{\La}}\b_{k-2r}=0.$$
		It implies that
		$$\mathcal{H}^{p,q}_{(2),d+d^{\La}}(X)=\bigoplus_{r\geq\max\{0,k-n\}}L^{r}P\mathcal{H}^{p-r,q-r}_{(2),d+d^{\La}}.$$
		Noting that $$P\mathcal{H}^{p-r,q-r}_{(2),d+d^{\La}}=P\mathcal{H}^{p-r,q-r}_{(2);J}.$$ 
		Following Weil formula, 
		$$\ast L^{r}\b_{k-2r}=(-1)^{\frac{(k-2r)(k-2r+1)}{2}}(\sqrt{-1})^{p-q} L^{n-k+r}\b_{k-2r}.$$
		Therefore, $L^{r}\b_{k-2r}\in\mathcal{H}^{p,q}_{(2)}(X)$. We get
		$$\mathcal{H}^{p,q}_{(2);d+d^{\La}}=\bigoplus_{r\geq\max\{0,k-n\}}L^{r}P\mathcal{H}^{p-r,q-r}_{(2),d+d^{\La}}=\bigoplus_{r\geq\max\{0,k-n\}}L^{r}P\mathcal{H}^{p-r,q-r}_{(2);J}\subset\mathcal{H}^{p,q}_{(2);J}(X).$$
		Hence the conclusion follows form Lemma \ref{L8}.
	\end{proof}
\begin{remark}
In \cite{CW1,CW2}, the authors extended the K\"{a}hler identities to the non-integrable setting. In fact, Proposition \ref{P7} is the generalized Hard Lefchetz Duality of the space of $(p,q)$-harmonic forms on compact almost K\"{a}hler manifolds, see \cite[Theorem 5.1]{CW1}.
\end{remark}
	\subsection{Symplectic parabolic}
	Let now $(X^{2n},J)$ be a almost K\"{a}hler manifold and $g$ be a Hermitian metric. Then according to Cao-Xavier \cite{CX} and Jost-Zuo \cite{JZ}, if $J$ is integrable and $\w$ is $d$(sublinear), then $\mathcal{H}^{p,q}_{(2)}(X)=\{0\}$, unless $k:=p+q=n$. In this section we will see that the same conclusions hold in the category of almost K\"{a}hler manifolds.
	\begin{proposition}\label{P3}
		Let $(X^{2n},J,\w)$ be a complete symplectic  manifold with a $d$(sublinear) symplectic form $\w$. Then any $\a\in P\mathcal{H}^{p,q}_{(2);J}(X)$ of degree $k:=p+q<n$ vanishes.
	\end{proposition}
	\begin{proof}
		For any $\a\in P^{k}$, $k:=p+q<n$, following  Weil formula (\ref{E5}), it implies that
		$$\ast\a= C(n,p,q)\a\wedge\w^{n-k},$$
		where $C(n,p,q)=\sqrt{-1}^{p-q}(-1)^{\frac{k(k+1)}{2}}\frac{1}{(n-k)!}$. By hypothesis, there exists a 1-form $\theta$ with $\w=d\theta$ and
		$$\|\theta(x)\|_{L_{\infty}}\leq c(1+\rho(x,x_{0})),$$
		where $c$ is an absolute constant. In what follows we assume that the distance function $\rho(x, x_{0})$ is smooth for $x\neq x_{0}$. The general case follows easily by an approximation argument. We observe that
		$$\ast\a=d\eta,$$ 
		where $$\eta=C(n,p,q)(\theta\wedge\a\wedge\w^{n-k-1})$$
		
		Let $h:\mathbb{R}\rightarrow\mathbb{R}$ be smooth, $0\leq h\leq1$,
		$$
		h(t)=\left\{
		\begin{aligned}
		1, &  & t\leq0 \\
		0,  &  & t\geq1
		\end{aligned}
		\right.
		$$
		and consider the compactly supported function
		$$f_{j}(x)=h(\rho(x_{0},x)-j),$$
		where $j$ is a positive integer.
		
		Noticing that $f_{j}\ast\a$ has compact support, one has
		\begin{equation}\label{E03}
		\begin{split}
		\langle \ast\a,f_{j}\ast\a\rangle_{L^{2}(X)}&=\langle d\eta,f_{j}\ast\a\rangle_{L^{2}(X)}\\
		&=\langle\eta,d^{\ast}(f_{j}\ast\a)\rangle_{L^{2}(X)}\\
		&=\langle\theta\wedge\a\wedge\w^{n-k-1}, d^{\ast}(f_{j}\ast\a)\rangle_{L^{2}(X)}\\
		&=\langle\theta\wedge\a\wedge\w^{n-k-1}, \ast(df_{j}\wedge\a)\rangle_{L^{2}(X)}.\\
		\end{split}
		\end{equation}
		Since $0\leq f_{j}\leq 1$ and $\lim_{j\rightarrow\infty}f_{j}(x)(\ast\a)(x)=\ast\a(x)$, it follows from the dominated convergence theorem that
		\begin{equation}\label{E04}
		\lim_{j\rightarrow\infty}\langle \ast\a,f_{j}\ast\a\rangle_{L^{2}(X)}=\|\a\|^{2}_{L^{2}(X)}.
		\end{equation}
		Since $\w$ is bounded, $supp(df_{j})\subset B_{j+1}\backslash B_{j}$ and $\|\theta(x)\|_{L_{\infty}}=O(\rho(x_{0},x))$, one obtains that
		\begin{equation}\label{E02}
		\langle\theta\wedge\a\wedge\w^{n-k-1}, \ast(df_{j}\wedge\a\rangle_{L^{2}(X)}\leq (j+1)C\int_{B_{j+1}\backslash B_{j}}|\a(x)|^{2}dx,
		\end{equation}
		where $C$ is a constant independent of $j$.
		
		We claim that there exists a subsequence $\{j_{i}\}_{i\geq1}$ such that
		\begin{equation}\label{E01}
		\lim_{i\rightarrow\infty}(j_{i}+1)\int_{B_{j_{i}+1}\backslash B_{j_{i}}}|\a(x)|^{2}dx=0.
		\end{equation}
		If not, there exists a positive constant $a$ such that
		$$\lim_{j\rightarrow\infty}(j+1)\int_{B_{j+1}\backslash B_{j}}|\a(x)|^{2}dx\geq a>0.$$
		This inequality implies
		\begin{equation}\nonumber
		\begin{split}
		\int_{X}|\a(x)|^{2}dx&=\sum_{j=0}^{\infty}\int_{B_{j+1}\backslash B_{j}}|\a(x)|^{2}dx\\
		&\geq a\sum_{j=0}^{\infty}\frac{1}{j+1}\\
		&=+\infty\\
		\end{split}
		\end{equation}
		which is a contradiction to the assumption $\int_{X}|\a(x)|^{2}dx<\infty$. Hence, there exists a subsequence $\{j_{i}\}_{i\geq1}$ for which (\ref{E01}) holds. Using (\ref{E02}) and (\ref{E01}), one obtains
		\begin{equation}\label{E05}
		\lim_{i\rightarrow\infty}\langle\theta\wedge\a\wedge\w^{n-k-1}, \ast(df_{j_{i}}\wedge\a)\rangle_{L^{2}(X)}=0
		\end{equation}
		It now follows from (\ref{E03}), (\ref{E04}) and (\ref{E05}) that $\|\a\|_{L^{2}(X)}=0$, i.e, $\a=0$. 
	\end{proof}
	Following the $L^{2}$-decomposition in Proposition \ref{P7}, we then have
	\begin{theorem}\label{T4}
		Let $(X^{2n},J,\w)$ be a complete almost K\"{a}hler manifold with a $d$(sublinear) symplectic form $\w$.  Then  
		$$\mathcal{H}_{(2);J}^{p,q}(X)=\{0\}$$ unless $k:=p+q=n$.
	\end{theorem}
	\begin{proof}
		The conclusion follows from Proposition \ref{P7} and \ref{P3}.
	\end{proof}
	Suppose that $J$ is integrable, i.e, $(X^{2n},J,\w)$ is a complete K\"{a}hler manifold. We have a $L^{2}$-decomposition for the space of the $L^{2}$-harmonic $k$-forms.
	\begin{lemma}\label{L2}(\cite{Huy})
		Let $(X^{2n},\w)$ be a complete K\"{a}hler manifold with a K\"{a}hler form $\w$. If $\a\in\Om^{k}_{0}(X)$, we denote $\a:=\sum_{p+q=k}\a_{p,q}$, $\a_{p,q}\in\Om^{p,q}_{0}(X)$, then we have
		$$\|d\a\|^{2}+\|d^{\ast}\a\|^{2}=\sum_{p+q=k}(\|d\a_{p,q}\|^{2}+\|d^{\ast}\a_{p,q}\|^{2}).$$
		In particular,  we have a decomposition for the space of the $L^{2}$-harmonic $k$-forms:
		$$\mathcal{H}^{k}_{(2)}(X)=\bigoplus_{p+q=k}\mathcal{H}^{p,q}_{(2)}(X).$$
	\end{lemma}
	\begin{proof}
		We denote $\De_{d}=dd^{\ast}+d^{\ast}d$ and $\De_{\bar{\pa}}=\bar{\pa}\bar{\pa}^{\ast}+\bar{\pa}^{\ast}\bar{\pa}$. We have an identity $\De_{d}=2\De_{\bar{\pa}}$. For any $\a\in\Om^{k}_{0}(X)$, we get
		\begin{equation*}
		\begin{split}
		\|d\a\|^{2}+\|d^{\ast}\a\|^{2}:&=\langle\De_{d}\a,\a\rangle_{L^{2}(X)}\\
		&=2\langle\De_{\bar{\pa}}\a,\a\rangle_{L^{2}(X)}\\
		&=2\langle\sum_{p,q}\De_{\bar{\pa}}\a_{p,q},\sum_{p,q}\a_{p,q}\rangle_{L^{2}(X)}\\
		&=2\sum_{p,q}\langle\De_{\bar{\pa}}\a_{p,q},\a_{p,q}\rangle_{L^{2}(X)}\\
		&=\sum_{p,q}\langle\De_{d}\a_{p,q},\a_{p,q}\rangle_{L^{2}(X)}\\
		&=\sum_{p,q}(\|d\a_{p,q}\|^{2}+\|d^{\ast}\a_{p,q}\|^{2}).\\
		\end{split}
		\end{equation*}
		Here we use the fact that $\De_{\bar{\pa}}\a_{p,q}$ is a $(p,q)$-form.
	\end{proof}
	Following the $L^{2}$-decomposition in Lemma \ref{L2}, we have
	\begin{corollary}(\cite{CX,JZ})\label{C2}
		Let $(X^{2n},\w)$ be a complete K\"{a}hler  manifold with a $d$(sublinear) K\"{a}hler form $\w$.  Then $$\mathcal{H}_{(2)}^{k}(X)=\{0\}$$ unless $k=n$.
	\end{corollary}
	\begin{proof}
		The conclusion follows form Lemma \ref{L2} and Theorem \ref{T4}.
	\end{proof}
	\subsection{Symplectic hyperbolic}
	In this section, we extend the idea of \cite[Theorem 3.7]{Hua} to the case of  symplectic manifold with a $d$(bounded) symplectic form $\w$. 
	\begin{proposition}\label{P1}
		Let $(X^{2n},J,\w)$ be a complete almost K\"{a}hler  manifold with a $d$(bounded) symplectic form $\w$, i.e., there exists a bounded $1$-form $\theta$ such that $\w=d\theta$. Then every $(p,q)$-form $\a\in P_{J}^{p,q}\cap\Om^{k}_{0}\subset P^{k}\cap\Om^{k}_{0}$ on $X$ of degree $k:=p+q<n$ satisfies the inequality
		\begin{equation}\label{E2}
		c_{n,k}\|\theta\|^{-2}_{L_{\infty}}\|\a\|^{2}_{L^{2}(X)}\leq\|d\a\|^{2}_{L^{2}(X)},
		\end{equation}
		where $c_{n,k}>0$ is a constant which depends only on $n,k$.
	\end{proposition}
	\begin{proof}
		Inequality (\ref{E2}) makes sense, strictly speaking, if $d\a$ (as well as $\a$) is in $L^{2}$. The linear map $L^{n-k}:\Om^{k}\rightarrow\Om^{2n-k}$ for $k\leq n-1$ is a bijective quasi-isometry on $P_{J}^{p,q}$ $(p+q=k)$, thus any $\a\in P_{J}^{p,q}$ satisfies
		$$\a=C(n,k)\ast L^{n-k}\a=C(n,k)\ast(\a\wedge\w^{n-k}),$$
		where $C(n,k)=\sqrt{-1}^{p-q}(-1)^{\frac{k(k+1)}{2}}\frac{1}{(n-k)!}$. We denote $\ast\a=d\eta-\tilde{\a}$, for $$\eta=C(n,k)(\theta\wedge\a\wedge\w^{n-k-1}),\ \tilde{\a}=C(n,k)(\theta\wedge d\a\wedge\w^{n-k-1})$$
		Observe that
		$\|\a\|_{L^{2}(X)}=\|\ast\a\|_{L^{2}(X)}$ and
		$\|\eta\|_{L^{2}(X)}\lesssim\|\theta\|_{L_{\infty}}\|\a\|_{L^{2}(X)}$. Now, we have
		\begin{equation}\label{E6}
		\begin{split}
		\|\a\|^{2}_{L^{2}(X)}&=\langle\ast\a,d\eta-\tilde{\a}\rangle_{L^{2}(X)}\\
		&\leq |\langle\ast\a,d\eta\rangle_{L^{2}(X)}|+|\langle\ast\a,\tilde{\a}\rangle|_{L^{2}(X)}\\
		&:=I_{1}+I_{2}.\\
		\end{split}
		\end{equation}
		The first term of the right hand:
		\begin{equation}\label{E7}
		\begin{split}
		I_{1}=|\langle\ast\a,d\eta\rangle_{L^{2}(X)}|&=|\langle\ast d\a,\eta\rangle_{L^{2}(X)}|\\
		&\leq\|d\a\|_{L^{2}(X)}\|\eta\|_{L^{2}(X)}\\
		&\lesssim\|d\a\|_{L^{2}(X)}\|\theta\|_{L_{\infty}}\|\a\|_{L^{2}(X)}\\
		\end{split}
		\end{equation}
		The second term of right hand:
		\begin{equation}\label{E10}
		\begin{split}
		I_{2}=|\langle\a,\tilde{\a}\rangle_{L^{2}(X)}|&\leq\|\a\|_{L^{2}(X)}\|\tilde{\a}\|_{L^{2}(X)}\\
		&\lesssim\|\a\|_{L^{2}(X)}\|\theta\|_{L{\infty}}\|d\a\|_{L^{2}(X)}.\\
		\end{split}
		\end{equation}
		Substituting (\ref{E7}) and  (\ref{E10}) into (\ref{E6}), it follows that 
		$$\|\a\|^{2}_{L^{2}(X)}\lesssim\|\a\|_{L^{2}(X)}\|\theta\|_{L_{\infty}}\|d\a\|_{L^{2}(X)}.$$
		Therefore, we obtain the inequality (\ref{E2}).
	\end{proof}
	\begin{lemma}\label{L11}
		If $B_{k}\in P^{k}$, $(k\leq n)$, then for any $0\leq j\leq i\leq (n-k) $, we have
		\begin{equation}\label{E4}
		\langle L^{i}B_{k},L^{i}A_{k}\rangle_{L^{2}(X)}=\frac{(n-k-i+j)!i!}{(n-k-i)!(i-j)!}\langle L^{i-j}B_{k},L^{i-j}A_{k}\rangle_{L^{2}(X)},\ \forall A_{k}\in\Om^{k}.
		\end{equation}
	\end{lemma}
	\begin{proof}
		If $\a\in\Om^{k}$, there is a  formula \cite{Huy} Corollary 1.2.28:
		$$[L^{i},\La]\a=i(k-n+i-1)L^{i-1}\a.$$
		Therefore, we have
		\begin{equation*}
		\begin{split}
		\langle L^{i}B_{k},L^{i}A_{k}\rangle_{L^{2}(X)}&=\langle[\La,L^{i}]B_{k},L^{i-1}A_{k}\rangle_{L^{2}(X)}\\
		&=i(n+1-(k+i))\langle L^{i-1}B_{k},L^{i-1}A_{k}\rangle_{L^{2}(X)}\\
		&=i(n+1-(k+i))\langle[\La,L^{i-1}]B_{k},B_{k}\rangle_{L^{2}(X)}\\
		&=i(n+1-(k+i))(i-1)(n+1-(k+i-1))\langle L^{i-2}B_{k},L^{i-2}A_{k}\rangle_{L^{2}(X)}\\
		&=\cdots\\
		&=\frac{(n-k-i+j)!i!}{(n-k-i)!(i-j)!}\langle L^{i-j}B_{k},L^{i-j}A_{k}\rangle_{L^{2}(X)}\\
		\end{split}
		\end{equation*}
		We complete this proof.
	\end{proof}
	Lefschetz decomposing $dB_{k}$, we can formally write
	$$dB_{k}=B^{0}_{k+1}+LB^{1}_{k+1}+\cdots+\frac{1}{r!}L^{r}B_{k+1-2r}+\cdots,$$
	But in fact the differential operators acting on primitive forms have
	special properties.
	\begin{lemma}(\cite[Lemma 2.4]{Tseng-Yau})
		Let $B_{k}\in P^{k}$ with $k\leq n$. The differential operators $(d,d^{\La})$ acting on $B_{k}$ take the following forms:\\
		(i) If $k<n$, then $dB_{k}=B^{0}_{k+1}+LB^{1}_{k-1}$;\\
		(ii) If $k=n$, then $dB_{k}=LB^{1}_{k-1}$;\\
		(iii) $d^{\La}B_{k}=-(n-k+1)B^{1}_{k-1}$.\\
		for some primitive forms $B^{0},B^{1}\in P^{\ast}$.
	\end{lemma} 
	\begin{lemma}\label{L10}
		If $\a\in\Om_{0}^{k}$, then
		$$\|d\a\|^{2}_{L^{2}(X)}+\|d^{\La}\a\|^{2}_{L^{2}(X)}\approx\sum_{r\geq0}\|d\b_{k-2r}\|^{2}_{L^{2}(X)}.$$
	\end{lemma}
	\begin{proof}
		We denote  $\a_{k}=\sum_{r\geq0}L^{r}\b_{k-2r}$, 
		where $\b_{k-2r}\in P^{k-2r}$. For convenience, we denote $\b_{k-2r}\equiv0$ for all $r>k/2$. By the operator $\mathcal{D}_{d+d^{\La}}$ communicates with $L$, we have
		\begin{equation}\label{E13}
		\begin{split}
		\|d\a\|^{2}_{L^{2}(X)}+\|d^{\La}\a\|^{2}_{L^{2}(X)}&=\langle\mathcal{D}_{d+d^{\La}}\a_{k},\a_{k}\rangle_{L^{2}(X)}\\
		&=\langle\sum_{r\geq0}L^{r}\mathcal{D}_{d+d^{\La}}\b_{k-2r},\sum_{r\geq0}L^{r}\b_{k-2r}\rangle_{L^{2}(X)}\\
		&=\sum_{p=q}\langle L^{p}\mathcal{D}_{d+d^{\La}}\b_{k-2p},L^{q}\b_{k-2q}\rangle_{L^{2}(X)}\\
		&+\sum_{p\neq q}\langle L^{p}\mathcal{D}_{d+d^{\La}}\b_{k-2p},L^{q}\b_{k-2q}\rangle_{L^{2}(X)}\\
		\end{split}
		\end{equation}
		The first term of the right hand in (\ref{E13}) satisfies
		\begin{equation*}
		\begin{split}
		\sum_{p=q}\langle L^{p}\mathcal{D}_{d+d^{\La}}\b_{k-2p},L^{q}\b_{k-2q}\rangle_{L^{2}(X)}&\approx\sum_{r\geq0}\langle \mathcal{D}_{d+d^{\La}}\b_{k-2r},\b_{k-2r}\rangle_{L^{2}(X)}\\
		&=\sum_{r\geq0}\|d\b_{k-2r}\|^{2}_{L^{2}(X)}+\|d^{\La}\b_{k-2r}\|^{2}_{L^{2}(X)}\\
		&\approx\sum_{r\geq0}\|d\b_{k-2r}\|^{2}_{L^{2}(X)}\\
		\end{split}
		\end{equation*}
		Nothing that the operator $\mathcal{D}_{d+d^{\La}}$ communicates with $\La$, see \cite[Lemma 3.7]{Tseng-Yau}. Then $\mathcal{D}_{d+d^{\La}}\b_{k-2p}\in P^{k-2p}$. For any $p\neq q$, following the Lemma \ref{L11}, we observe that
		\begin{equation*}
		\langle L^{p}\mathcal{D}_{d+d^{\La}}\b_{k-2p},L^{q}\b_{k-2q}\rangle_{L^{2}(X)}=0.
		\end{equation*}
		Therefore the second term of the right hand in (\ref{E13}) is zero.
	\end{proof}
	\begin{corollary}\label{P2}
		If $\a\in\Om_{J}^{p,q}$, $(p+q\leq n)$, we denote $\a=\sum_{r\geq0}\b_{k-2r}$, where $\b_{k-2r}\in P^{p-r,q-r}_{J}$, then
		$$\|d\a\|^{2}_{L^{2}(X)}+\|d^{\ast}\a\|^{2}_{L^{2}(X)}\approx\sum_{r\geq0}\|d\b_{k-2r}\|^{2}_{L^{2}(X)}.$$
		In particular,  we have a decomposition for the space of the $L^{2}$-harmonic forms of bi-degree $(p,q)$:
		$$\mathcal{H}^{p,q}_{(2);J}(X)=\bigoplus_{r\geq0} L^{r}P\mathcal{H}^{p-r,q-r}_{(2);J}(X).$$
	\end{corollary}
	\begin{proof}
		The conclusions follow from Lemma \ref{L8} and \ref{L10}.
	\end{proof}
	Now we can give a lower bound on the spectra of the Laplace operator $\De_{d}:=dd^{\ast}+d^{\ast}d$ on $L^{2}$-forms $\Om_{J}^{p,q}(X)$ for $p+q\neq n$ to sharpen the vanishing theorem in $d$(bounded) case. 
	\begin{theorem}\label{T3}
		Let $(X^{2n},J,\w)$ be a complete almost K\"{a}hler  manifold with a $d$(bounded) symplectic form $\w$, i.e., there exists a bounded $1$-form $\theta$ such that $\w=d\theta$. Then any $\a\in \Om_{J}^{p,q}(X)\cap\Om^{k}_{0}(X)$ on $X$ of degree $k:=p+q\neq n$ satisfies the inequality
		\begin{equation*}
		c_{n,k}\|\theta\|^{-2}_{L_{\infty}}\|\a\|^{2}_{L^{2}(X)}\leq\|d\a\|^{2}_{L^{2}(X)}+\|d^{\ast}\a\|^{2}_{L^{2}(X)},
		\end{equation*}
		where $c_{n,k}>0$ is a constant which depends only on $n,k$. In particular, 
		$$\mathcal{H}^{p,q}_{(2);J}(X)=\{0\}$$
		 unless $k:=p+q=n$.
	\end{theorem}
	\begin{proof}
		We only need consider $k<n$ case. The case $k>n$ follows by the Poincar\'{e} duality as the operator $\ast:\Om^{k}\rightarrow\Om^{2n-k}$ commutes with $\De_{d}$ and is isometric for the $L^{2}$-norms. Now we denote $\a=\sum_{r\geq0}L^{r}\b_{k-2r}$, where $\b_{k-2r}\in P_{J}^{i-r,j-r}$. We then have
		$$\|\a\|^{2}_{L^{2}(X)}=\sum_{r\geq0}\|L^{r}\b_{k-2r}\|^{2}_{L^{2}(X)}\lesssim\sum_{r\geq0}\|\b_{k-2r}\|^{2}_{L^{2}(X)}.$$ 
		Following Proposition \ref{P1}, it implies that
		$$\|\b_{k-2r}\|_{L^{2}(X)}\|\theta\|^{-1}_{L_{\infty}}\lesssim\|d\b_{k-2r}\|_{L^{2}(X)}.$$ 
		Following Corollary \ref{P2}, it implies that
		\begin{equation}\nonumber
		\begin{split}
		\|\a\|^{2}_{L^{2}(X)}&\lesssim\|\theta\|^{2}_{L_{\infty}}\sum_{r\geq0}\|d\b_{k-2r}\|^{2}_{L^{2}(X)}\\
		&\lesssim\|\theta\|^{2}_{L_{\infty}}(\|d\a\|^{2}_{L^{2}(X)}+\|d^{\ast}\a\|^{2}_{L^{2}(X)}).
		\end{split}
		\end{equation}
		We complete the proof.
	\end{proof}
	We can obtain a well-known result proved by Gromov, see \cite[1.4.A. Theorem]{Gro} or \cite{Pansu}.
	\begin{corollary}\label{C1}
		Let $(X^{2n},\w)$ be a complete K\"{a}hler  manifold with a $d$(bounded) K\"{a}hler form $\w$, i.e., there exists a bounded $1$-form $\theta$ such that $\w=d\theta$. Then any $\a\in\Om^{k}_{0}(X)$ satisfies the inequality
		\begin{equation*}
		c_{k}\|\theta\|^{-2}_{L_{\infty}}\|\a\|^{2}_{L^{2}(X)}\leq\|d\a\|^{2}_{L^{2}(X)}+\|d^{\ast}\a\|^{2}_{L^{2}(X)},
		\end{equation*}
		where $c_{k}>0$ is a constant which depends only on $n,k$. In particular, $$\mathcal{H}^{k}_{(2)}(X)=\{0\}$$
		unless $k=n$.
	\end{corollary}
	\begin{proof}
		Following Lemma \ref{L2} and Theorem \ref{T3}, we have
		\begin{equation*}
		\begin{split}
		\|d\a\|^{2}+\|d^{\ast}\a\|^{2}:&=\sum_{p,q}(\|d\a_{p,q}\|^{2}+\|d^{\ast}\a_{p,q}\|^{2})\\
		&\geq \sum_{p,q}c_{n,p,q}\|\theta\|^{-2}_{L_{\infty}}\|\a_{p,q}\|^{2}_{L^{2}(X)}\\
		&\geq \min_{p,q}c_{n,p,q}\|\theta\|^{-2}_{L_{\infty}}\sum_{p,q}\|\a_{p,q}\|^{2}_{L^{2}(X)}\\
		&\geq \min_{p,q}c_{n,p,q}\|\theta\|^{-2}_{L_{\infty}}\|\a\|^{2}_{L^{2}(X)}.\\
		\end{split}
		\end{equation*}
		Here we use that fact that $\|\a\|^{2}_{L^{2}(X)}=\sum_{p,q}\|\a_{p,q}\|^{2}_{L^{2}(X)}$.
	\end{proof}
\subsection{$L^{2}$-decomposition for almost K\"{a}hler manifolds }
Let $(X,J,\w)$ be a $2n$-dimensional almost K\"{a}hler manifold. Then $J$ acts as an involution on the space of smooth $2$-forms $\Om^{2}(X)$:  given $\a\in\Om^{2}(X)$, for every pair of vector fields $u,v$ on $X$
$$J\a(u,v)=\a(Ju,Jv).$$
Therefore the space $\Om^{2}(X)$ splits as the direct sum of $\pm1$-eigenspaces $\Om^{\pm}$, i.e., $\Om^{2}(X)=\Om^{+}(X)\oplus\Om^{-}(X)$. Let us denote by $\mathcal{Z}^{2}_{(2)}(X)$ the space of closed $2$-forms which are in $L^{2}$ and set
$$\mathcal{Z}^{\pm}_{(2);J}=\mathcal{Z}^{2}_{(2)}(X)\cap\Om^{\pm}_{J}.$$
Define 
$$H^{\pm}_{(2);J}:=\{\mathfrak{a}\in H^{2}_{(2)}(X):\exists\a\in\mathcal{Z}^{\pm}_{(2);J}\ such\ that\ \mathfrak{a}=[\a]\}.$$
For closed almost complex $4$-manifolds Dr\v{a}ghici-Li-Zhang showed in \cite{DLZ} that there is a direct sum decomposition
$$H^{2}_{dR}(X;\mathbb{R})=H^{+}(X)\oplus H^{-}(X).$$
In \cite{HT}, Hind-Tomassini generalized such a decomposition to the $L^{2}$ setting.
\begin{theorem}\cite[Theorem 4.8]{HT}
	Let $(X,J,\w)$ be a complete almost K\"{a}hler $4$-dimensional manifold. Then, we have the following decomposition
	$$H_{(2)}^{2}(X)=\overline{H^{+}_{(2);J}(X)}\oplus\overline{H^{-}_{(2);J}(X)}.$$
\end{theorem}
Given any J-anti-invariant form $\a$ on a 2n-dimensional almost Hermitian
manifold $X$, we have 
$$\ast\a^{-}=\frac{1}{(n-2)!}\a^{-}\wedge\w^{n-2}.$$
We  then have
\begin{proposition}(\cite[Corollary 4.1]{HT})
Closed anti-invariant forms are harmonic, that is, we have an inclusion $\mathcal{Z}^{-}_{(2);J}\hookrightarrow\mathcal{H}^{2}_{(2)}$.
\end{proposition}\label{P4}
Following the idea in Proposition \ref{P3}, we get
\begin{corollary}
Let $(X^{2n},J,\w)$ be a complete  $2n$-dimensional almost K\"{a}hler manifold with a $d$(sublinear) symplectic form $\w$, ($n\geq3$). Then $$H^{-}_{(2);J}(X)=\{0\}.$$
\end{corollary}
\begin{proof}
We denote by $\a^{-}$ the harmonic anti-invariant form on $X$. Noticing that $\a^{-}$ could be a sum of terms of type $(2,0)$ and $(0,2)$, it implies that $\La_{\w}\a^{-}=0$, i.e., $\a\in \mathcal{H}^{2}_{(2);J}(X)\cap\ker{P}$. Even though $\mathcal{H}^{-}_{(2);J}(X)\subset\mathcal{H}^{2,0}_{(2);J}(X)\oplus\mathcal{H}^{0,2}_{(2);J}(X)$ does not hold, for any harmonic anti-invariant form $\a^{-}$ we have the identity 
	$$\ast\a^{-}=\frac{1}{(n-1)!}\a^{-}\wedge\w^{n-2}=\frac{1}{(n-1)!}d(\a^{-}\wedge\theta\wedge\w^{n-3}):=d\eta,$$
	where $$\eta=\frac{1}{(n-1)!}(\a^{-}\wedge\theta\wedge\w^{n-3}).$$
	We follow the method in Proposition \ref{P3} to choose the compactly supported function $$f_{j}(x)=h(\rho(x_{0},x)-j),$$ 
	where $j$ is positive integer. Noticing that $f_{j}\ast\a$ has compact support, one has
	\begin{equation*}
	\begin{split}
	\langle \ast\a^{-},f_{j}\ast\a^{-}\rangle_{L^{2}(X)}&=\langle d\eta,f_{j}\ast\a\rangle_{L^{2}(X)}\\
	&=\langle\theta\wedge\a^{-}\wedge\w^{n-3}, \ast(df_{j}\wedge\a^{-}\rangle_{L^{2}(X)}.\\
	\end{split}
	\end{equation*}
	Using the idea in Proposition \ref{P3}, we obtain that there exists a subsequence $\{j_{i}\}_{i\geq1}$ such that
	$$\langle\theta\wedge\a^{-}\wedge\w^{n-3}, \ast(df_{j_{i}}\wedge\a^{-}\rangle_{L^{2}(X)}\rightarrow 0,\ as\ i\rightarrow\infty.$$
	Therefore, we have
	$$
	0=\lim_{i\rightarrow\infty}\langle \ast\a^{-},f_{j_{i}}\ast\a^{-}\rangle_{L^{2}(X)}=\|\a^{-}\|^{2}_{L^{2}(X)},
	$$
	i.e., $\a^{-}=0$. We complete this proof.
\end{proof}
Let $\a^{+}_{J}$ be a $J$-self-dual-invariant form in $X^{2n}$. Then we can denote $\a^{+}_{J}=f\w+\a_{J,0}^{1,1}$, where $f\in\Om^{0}(X)$ and $\a_{J,0}^{1,1}\in P^{1,1}$.
\begin{corollary}
Let $(X^{2n},J,\w)$ be a complete $2n$-dimensional almost K\"{a}hler manifold with a $d$(bounded) symplectic form $\w$, ($n\geq3$). If $\a^{+}_{J}\in H^{+}_{(2);J}(X)$, then there is a positive constant $c_{n}=c(n)$ such that
\begin{equation*}
c_{n}\|\theta\|^{-2}_{L_{\infty}}\|\a^{+}_{J}\|^{2}_{L^{2}(X)}\leq\|df\|^{2}_{L^{2}(X)},
\end{equation*}
Furthermore, if $\a^{+}_{J}$ also satisfies $d^{\La}\a^{+}_{J}=0$, then
$\a^{+}_{J}=0$.
\end{corollary}
\begin{proof}
If $\a^{+}_{J}\in H^{+}_{(2);J}(X)$, i.e., $d\a^{+}_{J}=0$. Therefore, we have
$$df\wedge\w+d\a_{J,0}^{1,1}=0\Rightarrow \|df\|^{2}=(n-1)\|d\a^{1,1}_{J,0}\|^{2}.$$
Following the Proposition \ref{P1}, we get
$$\|\theta\|^{-2}_{L_{\infty}}\|f\|^{2}_{L^{2}(X)}\lesssim\|df\|^{2}_{L^{2}(X)}$$
and
$$\|\theta\|^{-2}_{L_{\infty}}\|\a^{1,1}_{J,0}\|^{2}_{L^{2}(X)}\lesssim\|d\a^{1,1}_{J,0}\|^{2}_{L^{2}(X)}.$$
Therefore, we get
$$c_{n}\|\theta\|^{-2}_{L_{\infty}}\|\a^{+}_{J}\|^{2}_{L^{2}(X)}\leq\|df\|^{2}_{L^{2}(X)}.$$
Noticing that $d^{\La}\a^{+}_{J}=d\La\a_{J}^{+}=ndf$. Therefore, if $\a^{+}_{J}\in\ker d^{\La}$, i.e. $df=0$, then
$\a^{+}_{J}=0$.
\end{proof}
\section*{Acknowledgements}
We would like to thank the anonymous referees for careful reading of my manuscript and helpful comments. We would like to thank Professor H.Y. Wang for drawing our attention to the symplectic parabolic manifold and generously helpful suggestions about these. This work was supported in part by NSF of China (11801539) and the Fundamental Research Funds of the Central Universities (WK3470000019), the USTC Research Funds of the Double First-Class Initiative (YD3470002002). Part of this article was completed when the author visited the Chern Institute of Mathematics in Nov. 2018. The author would like to thank the Institute for the hospitality.

\end{document}